\RequirePackage{etoolbox}
\csdef{input@path}{%
{sty/},
{img/},
{bib/}
}
\documentclass[numbers,compress]{vmsta}

\volume{0}
\issue{0}
\pubyear{2020}
\articletype{research-article}

\newtheorem{thm}{Theorem}
\theoremstyle{definition}
\newtheorem{defin}{Definition}
\newtheorem{remark}{Remark}
\newtheorem{lemma}{Lemma}
\newtheorem{cor}{Corollary}
\newtheorem{assumption}{Assumption}
\newcommand{\eps}{\varepsilon}
\newcommand{\R}{\mathbb{R}}
\newcommand{\pr}{\mathrm{P}}
\newcommand{\E}{\mathrm{E}}
\newcommand{\la}{\lambda}

\begin{document}

\begin{frontmatter}

\pretitle{Research Article}

\title{Long-time behavior of a non-autonomous stochastic predator-prey model with jumps}

\author[a]{\inits{Olg.}\fnms{Olga}~\snm{Borysenko}\thanksref{}\ead[label=e1]{olga\_borisenko@ukr.net}\orcid{0000-0000-0000-0000}}

\author[b]{\inits{O.}\fnms{Oleksandr}~\snm{Borysenko}\thanksref{cor1}\ead[label=e2]{odb@univ.kiev.ua}}
\thankstext[id=cor1]{Corresponding author.}
\address[a]{\institution{Department of Mathematical Physics,
National Technical University of Ukraine}, 37, Prosp.Peremohy, Kyiv, 03056, \cny{ Ukraine
}}

\address[b]{\institution{Department of Probability Theory, Statistics and Actuarial
Mathematics, Taras Shevchenko National University of Kyiv , Ukraine}, 64 Volodymyrska Str., Kyiv, 01601 \cny{Ukraine}}

\begin{abstract}
It is proved the existence and uniqueness of the global positive solution to the system of stochastic differential equations describing a non-autonomous stochastic predator-prey model with modified version of Leslie-Gower term and Holling-type II functional response disturbed by white noise, centered and non-centered Poisson noises. We obtain sufficient conditions of stochastic ultimate boundedness, stochastic permanence, non-persistence in the mean, weak persistence in the mean and extinction of the solution to the considered system.
\end{abstract}

\begin{keywords}
\kwd{Stochastic Predator-Prey Model}
\kwd{Leslie-Gower and Holling-type II functional response}
\kwd{Global Solution}
\kwd{Stochastic Ultimate Boundedness}
\kwd{Stochastic Permanence}
\kwd{Extinction}
\kwd{Non-Persistence in the Mean}
\kwd{Weak Persistence in the Mean}
\end{keywords}

\begin{keywords}[MSC2010]%
\kwd{92D25}
\kwd{60H10}
\kwd{60H30}
\end{keywords}

\end{frontmatter}

\section{Introduction}\label{}

The deterministic predator-prey model with modified version of Leslie-Gower term and Holling-type II functional response is studied in \cite{Aziz}. This model has a form
\begin{align}\label{eq1}
dx_1(t)=x_1(t)\left(a-bx_1(t)-\frac{cx_2(t)}{m_1+x_1(t)}\right)dt,\nonumber \\
dx_2(t)=x_2(t)\left(r-\frac{fx_2(t)}{m_2+x_1(t)}\right)dt,\end{align}
where $x_1(t)$ and $x_2(t)$ are the prey and predator population densities at time $t$, respectively. Positive constants $a,b,c,r,f,m_1,m_2$ defined as follows: $a$ is the growth rate of prey $x_1$; $b$ measures the strength of competition among individuals of species $x_1$; $c$ is the maximum value of the per capita reduction rate of $x_1$ due to $x_2$; $m_1$ and $m_2$ measure the extent to which the environment provides protection to prey $x_1$ and to the predator $x_2$, respectively; $r$ is the growth rate of predator $x_2$, and $f$ has a similar meaning to $c$. In \cite{Aziz} the authors study boundedness and global stability of the positive equilibrium of the model $(\ref{eq1})$.

In the papers \cite{Ji1}, \cite{Ji2}, \cite{Lin}  it is considered the stochastic version of model $(\ref{eq1})$ in the following form
\begin{align}\label{eq2}
dx_1(t)=x_1(t)\left(a-bx_1(t)-\frac{cx_2(t)}{m_1+x_1(t)}\right)dt+\alpha x_1(t)dw_1(t),\nonumber\\
dx_2(t)=x_2(t)\left(r-\frac{fx_2(t)}{m_2+x_1(t)}\right)dt+\beta x_2(t)dw_2(t),\end{align}
where $w_1(t)$ and $w_2(t)$ are mutually independent Wiener processes in \cite{Ji1}, \cite{Ji2}, and processes $w_1(t),w_2(t)$ are correlated in \cite{Lin}. In \cite{Ji1}  the authors proved that there is a unique positive solution to the system $(\ref{eq2})$, obtaining the sufficient conditions for extinction and persistence in the mean of predator and prey. In \cite{Ji2} it is shown that, under appropriate conditions, there is a stationary distribution of the solution to the system $(\ref{eq2})$ which is ergodic. In \cite{Lin} the authors prove that the densities of the distributions of the solution to the system $(\ref{eq2})$ can converges in $L^1$ to an invariant density or can converge weakly to a singular measure under appropriate conditions.

Population systems may suffer abrupt environmental perturbations,
such as epidemics, fires, earthquakes, etc. So it is natural to introduce Poisson noises into the population model for
describing such discontinuous systems.

In this paper, we consider the non-autonomous predator-prey model with modified version of Leslie-Gower term and Holling-type II functional response, disturbed by white noise and jumps generated by centered and non-centered Poisson measures. So, we take into account not only ``small'' jumps, corresponding to the centered Poisson measure, but also the ``large'' jumps, corresponding to the non-centered Poisson measure. This model is driven by the system of stochastic differential equations
\begin{align}\label{eq3}
dx_i(t)=x_i(t)\left[a_{i}(t)-b_{i}(t)x_{i}(t)-\frac{c_{i}(t)x_2(t)}{m(t)+x_{1}(t)}\right]dt+\sigma_i(t)x_i(t)dw_i(t)\nonumber\\
+\int\limits_{\mathbb{R}}\gamma_i(t,z)x_i(t)\tilde\nu_1(dt,dz)+\int\limits_{\mathbb{R}}\delta_i(t,z)x_i(t)\nu_2(dt,dz),\nonumber\\ x_i(0)=x_{i0}>0,\ i=1,2.
\end{align}
where $x_1(t)$ and $x_2(t)$ are the prey and predator population densities at time $t$, respectively, $b_2(t)\equiv0$, $w_i(t), i=1,2$ are independent standard one-dimensional Wiener processes, $\nu_i(t,A), i=1,2$ are independent Poisson measures, which are independent on $w_i(t),i=1,2$, $\tilde\nu_1(t,A)=\nu_1(t,A)-t\Pi_1(A)$,  $E[\nu_i(t,A)]=t\Pi_i(A),i=1,2$, $\Pi_i(A), i=1,2$ are finite measures on the Borel sets $A$ in $\mathbb{R}$.

To the best of our knowledge, there have been no papers devoted to the dynamical properties of the stochastic predator-prey model $(\ref{eq3})$, even in the case of centered Poisson noise. It is worth noting that the impact of centered and non-centered Poisson noises to the stochastic non-autonomous logistic model and to the stochastic two-species mutualism model is studied in the papers \cite{Bor1} -- \cite{Bor3}.

In the following we will use the notations $X(t)=(x_1(t),x_2(t))$, $X_0=(x_{10},x_{20})$, $|X(t)|=\sqrt{x_1^2(t)+x_2^2(t)}$, $\R^2_{+}=\{X\in\R^2:\ x_1>0,x_2>0\}$,
\begin{align}\alpha_i(t)=a_{i}(t)+\int_{\R}\delta_{i}(t,z)\Pi_2(dz),\nonumber\\
\beta_i(t)\!=\!\frac{\sigma^2_i(t)}{2}\!+\!\!\!\int\limits_{\R}[\gamma_i(t,z)\!-\!\ln(1\!+\!\gamma_i(t,z))]\Pi_1(dz)\!
-\!\!\!\int\limits_{\R}\ln(1\!+\!\delta_i(t,z))\Pi_2(dz),
\nonumber
\end{align} $i=1,2$. For bounded, continuous functions $f_i(t), t\in[0,+\infty), i=1,2$, let us denote
\begin{align}f_{i\sup}=\sup_{t\ge0}f_i(t), f_{i\inf}=\inf_{t\ge0}f_i(t), i=1,2,\nonumber\\ f_{\max}=\max\{f_{1\sup},f_{2\sup}\}, f_{\min}=\min\{f_{1\inf},f_{2\inf}\}.\nonumber\end{align}

We prove that system $(\ref{eq3})$ has a unique, positive, global (no explosion
in a finite time) solution for any positive initial value, and that this solution is stochastically ultimately bounded. The sufficient conditions for stochastic permanence, non-persistence in the mean, weak persistence in the mean and extinction of solution are derived.

The rest of this paper is organized as follows. In Section 2, we prove the existence of the unique global positive solution to the system $(\ref{eq3})$ and derive some auxiliary results. In Section 3, we prove the stochastic ultimate boundedness of the solution to the system $(\ref{eq3})$, obtainig conditions under which the solution is stochastically permanent. The sufficient conditions for non-persistence in the mean, weak persistence in the mean and extinction of the solution are derived.

\section{Existence of global solution and some auxiliary lemmas}

Let $(\Omega,{\cal F},\pr)$ be a probability space, $w_i(t), i=1,2, t\ge0$ are independent standard one-dimensional Wiener processes on $(\Omega,{\cal F},\pr)$, and $\nu_i(t,A), i=1,2$ are independent Poisson measures defined on $(\Omega,{\cal F},\pr)$ independent on $w_i(t), i=1,2$. Here $\E[\nu_i(t,A)]=t\Pi_i(A), i=1,2$, $\tilde\nu_i(t,A)=\nu_i(t,A)-t\Pi_i(A), i=1,2$, $\Pi_i(\cdot), i=1,2$ are finite measures on the Borel sets in $\mathbb{R}$. On the probability space $(\Omega,{\cal F},\pr)$ we consider an increasing, right continuous family of complete sub-$\sigma$-algebras $\{{\cal F}_{t}\}_{t\ge0}$, where ${\cal F}_{t}=\sigma\{w_i(s),\nu_i(s,A), s\le t,i=1,2\}$.

We need the following assumption.

\begin{assumption} \label{ass1} It is assumed, that $a_{i}(t), b_1(t)$, $c_i(t), \sigma_i(t), \gamma_{i}(t,z), \delta_{i}(t,z), i=1,2$, $m(t)$ are bounded, continuous on $t$ functions, $a_{i}(t)>0,i=1,2$, $b_{1\inf}>0$, $c_{i\inf}>0, i=1,2$, $m_{\inf}=\inf_{t\ge0}m(t)>0$, and $\ln(1+\gamma_i(t,z)),\ln(1+\delta_i(t,z)), i=1,2$ are bounded, $\Pi_i(\R)<\infty, i=1,2$.\end{assumption}

In what follows we will assume that Assumption \ref{ass1} holds.
\begin{thm} \label{thm1}
There exists a unique global solution $X(t)$ to the system $(\ref{eq3})$ for any initial value $X(0)=X_0\in \R^2_{+}$, and
$\pr\{X(t)\in \R^2_{+}\}=1$.
\end{thm}

\begin{proof}
Let us consider the system of stochastic differential equations
\begin{align}\label{eq4}
d\xi_i(t)=\left[a_{i}(t)-b_{i}(t)\exp\{\xi_{i}(t)\}-\frac{c_{i}(t)\exp\{\xi_{2}(t)\}}{m(t)+\exp\{\xi_1(t)\}}-\beta_i(t)\right]dt\nonumber\\+\sigma_i(t)dw_i(t)
+\int\limits_{\mathbb{R}}\ln(1+\gamma_i(t,z))\tilde\nu_1(dt,dz)+\int\limits_{\mathbb{R}}\ln(1+\delta_i(t,z))\tilde\nu_2(dt,dz),\nonumber\\ v_i(0)=\ln x_{i0},\ i=1,2.
\end{align}
The coefficients of the system $(\ref{eq4})$ are local Lipschitz continuous. So, for any initial value $(\xi_1(0),\xi_2(0))$ there exists a unique local solution $\Xi(t)=(\xi_1(t),\xi_2(t))$ on $[0,\tau_{e})$, where $\sup_{t<\tau_{e}}|\Xi(t)|=+\infty$ (cf. Theorem 6, p.246, \cite{GikhSkor}). Therefore, from the It\^{o} formula we derive that the process $X(t)=(\exp\{\xi_1(t)\},\linebreak\exp\{\xi_2(t)\})$ is a unique, positive local solution to the system (\ref{eq3}). To show this solution is global, we need to show that $\tau_{e}=+\infty$ a.s. Let $n_0\in\mathbb{N}$ be sufficiently large for $x_{i0}\in[1/n_0,n_0], i=1,2$. For any $n\ge n_0$ we define the stopping time
\begin{align}
\tau_n=\inf\left\{t\in[0,\tau_e):\ X(t)\notin\left(\frac{1}{n},n\right)\times\left(\frac{1}{n},n\right)\right\}.\nonumber
\end{align}
 It is easy to see that $\tau_n$ is increasing as $n\to+\infty$. Denote $\tau_{\infty}=\lim_{n\to\infty}\tau_n$, whence $\tau_{\infty}\le\tau_e$ a.s. If we prove that $\tau_\infty=\infty$ a.s., then $\tau_e=\infty$ a.s. and $X(t)\in \R^2_{+}$ a.s. for all $t\in[0,+\infty)$. So we need to show that $\tau_\infty=\infty$ a.s. If it is not true, there are constants $T>0$ and $\varepsilon\in(0,1)$, such that $\pr\{\tau_{\infty}<T\}>\varepsilon$. Hence, there is $n_1\ge n_0$ such that
\begin{align}\label{eq5}\pr\{\tau_{n}<T\}>\varepsilon,\quad \forall n\ge n_1.\end{align}
For the non-negative function $V(X)=\sum\limits_{i=1}^2 k_{i}(x_i-1-\ln x_i)$, $X=(x_1,x_2)$, $x_i>0$, $k_{i}>0$, $i=1,2$
by the It\^{o} formula we obtain
\begin{align}\label{eq6}
dV(X(t))=\sum_{i=1}^2 k_i \left\{\rule{0pt}{20pt}(x_i(t)-1)\left[a_i(t)-b_i(t)x_i(t)-\frac{c_{i}(t)x_2(t)}{m(t)+x_{1}(t)}\right]\right.\nonumber\\ \left.+\beta_i(t)+\int\limits_{\R}\delta_i(t,z)x_i(t)\Pi_2(dz)\right\}dt+\sum_{i=1}^2 k_i\left\{\rule{0pt}{20pt}(x_i(t)-1)\sigma_i(t)dw_i(t)\right.\nonumber\\
+\int\limits_{\R}\![\gamma_i(t,z)x_i(t)-\ln(1+\gamma_i(t,z))]\tilde\nu_1(dt,dz)\nonumber\\ \left.+
\int\limits_{\R}[\delta_i(t,z)x_i(t)-\ln(1+\delta_i(t,z))]\tilde\nu_2(dt,dz)\right\}.
\end{align}
Let us consider the function $f(t,x_1,x_2)=\phi(t,x_1)+\psi(t,x_1,x_2)$,  $x_1>0,\linebreak x_2>0$ where
\begin{align}
\phi(t,x_1)=-k_1 b_1(t)x_1^2+k_1\left(\rule{0pt}{12pt}\alpha_1(t)+b_1(t)\right)x_1+k_1\beta_1(t)+k_2\beta_2(t)\nonumber\\-k_1 a_1(t)-k_2a_2(t),\nonumber\\
\psi(t,x_1,x_2)=(m(t)+x_1)^{-1}\left[\rule{0pt}{12pt}-k_2c_2(t)x_2^2+\left(\rule{0pt}{12pt}k_2\alpha_2(t)-k_1 c_1(t)\right)x_1x_2\right.\nonumber\\ \left.+\left(\rule{0pt}{12pt}k_2\alpha_2(t)m(t)+k_1c_1(t)+k_2c_2(t)\right)x_2\rule{0pt}{12pt}\right].\nonumber
\end{align}
Under Assumption \ref{ass1} there is a constant $L_1(k_1,k_2)>0$, such that
\begin{align}
\phi(t,x_1)\le k_1\left[\rule{0pt}{12pt}-b_{1\inf}x_1^2+\left(\alpha_{1\sup}+b_{1\sup}\right)x_1\right]+\beta_{\max}(k_1+k_2)\le L_1(k_1,k_2).\nonumber
\end{align}
If $\alpha_{2\sup}\le0$, then for the function $\psi(t,x_1,x_2)$ we have
\begin{align}
\psi(t,x_1,x_2)\le \frac{-k_2 c_{2\inf}x_2^2+(k_1+k_2)c_{\max}x_2}{m(t)+x_1}\le L_2(k_1,k_2).\nonumber
\end{align}
If $\alpha_{2\sup}>0$, then for $k_2=k_1\frac{c_{1\inf}}{\alpha_{2\sup}}$ there is a constant $L_3(k_1,k_2)>0$, such that
\begin{align}
\psi(t,x_1,x_2)\le \left\{\rule{0pt}{12pt}-k_2 c_{2\inf}x_2^2+(k_2\alpha_{2\sup}-k_1 c_{1\inf})x_1x_2+\left[\rule{0pt}{12pt}k_2\alpha_{2\sup}m_{\sup}\right.\right.\nonumber\\ +\left.\left.\rule{0pt}{12pt}(k_1+k_2)c_{\max}\right]x_2\right\}(m(t)+x_1)^{-1}=
\frac{k_1}{m(t)+x_1}\left\{-\frac{c_{1\inf}c_{2\inf}}{\alpha_{2\sup}}x_2^2\right.\nonumber\\
\left.+\left[\rule{0pt}{14pt}c_{1\inf}m_{\sup}+\left(1+\frac{c_{1\inf}}{\alpha_{2\sup}}\right)c_{\max}\right]x_2\right\}\le L_3(k_1,k_2).\nonumber
\end{align}
Therefore there is a constant $L(k_1,k_2)>0$, such that $f(t,x_1,x_2)\le L(k_1,k_2)$.
So from $(\ref{eq6})$ we obtain by integrating
\begin{align}\label{eq7}
V(X(T\wedge\tau_n))\le V(X_0)+L(k_1,k_2)(T\wedge\tau_n)\nonumber\\+\sum_{i=1}^2k_{i}\left\{\int\limits_0^{T\wedge\tau_n}(x_i(t)-1)\sigma_i(t)dw_i(t)+
\int\limits_0^{T\wedge\tau_n}\!\!\!\int\limits_{\R}\left[\gamma_i(t,z)x_i(t)-\ln(1\right.\right.\nonumber\\ \left.\left.+\gamma_i(t,z))\right]\tilde\nu_1(dt,dz)+
\int\limits_0^{T\wedge\tau_n}\!\!\!\int\limits_{\R}[\delta_i(t,z)x_i(t)-\ln(1+\delta_i(t,z))]\tilde\nu_2(dt,dz)\right\}.
\end{align}
Taking expectations we derive from $(\ref{eq7})$
\begin{align}\label{eq8}
\E\left[V(X(T\wedge\tau_n))\right]\le V(X_0)+L(k_1,k_2)T.
\end{align}
Set $\Omega_n=\{\tau_n\le T\}$ for $n\ge n_1$. Then by $(\ref{eq5})$, $\pr(\Omega_n)=\pr\{\tau_n\le T\}>\varepsilon$, $\forall n\ge n_1$. Note that for every $\omega\in\Omega_n$ there is at least one of $x_1(\tau_n,\omega)$ and $x_2(\tau_n,\omega)$ equals either $n$ or $1/n$. So
\begin{align}
V(X(\tau_n))\ge\min\{k_1,k_2\} \min\{n-1-\ln n,\frac{1}{n}-1+\ln n\}.\nonumber
\end{align}
From $(\ref{eq8})$ it follows 
\begin{align}
V(X_0)+L(k_1,k_2)T\ge \E[\mathbf{1}_{\Omega_n}V(X(\tau_n))]\nonumber\\ \ge\varepsilon\min\{k_1,k_2\}\min\{n-1-\ln n,\frac{1}{n}-1+\ln n\},\nonumber
\end{align}
where $\mathbf{1}_{\Omega_n}$ is the indicator function of $\Omega_n$. Letting $n\to\infty$ leads to the contradiction $\infty>V(X_0)+L(k_1,k_2)T=\infty$. This completes the proof of the theorem.
\end{proof}

\begin{lemma}\label{lm1}
The density of prey population $x_1(t)$ obeys 
\begin{align}
\limsup_{t\to\infty}\frac{\ln(m+x_1(t))}{t}\le0,\ \forall m>0, \qquad\hbox{a.s.}
\end{align}
\end{lemma}

\begin{proof}
By It\^{o} formula for the process $e^t\ln(m+x_1(t))$ we have
\begin{align}\label{eq10}
e^{t}\ln(m+x_1(t))-\ln(m+x_{10})=\int_{0}^{t}e^s\left\{\rule{0pt}{14pt}\ln(m+x_1(s))\right.\nonumber\\+\!\frac{x_1(s)}{m+x_1(s)}\left[\rule{0pt}{14pt}
a_1(s)\!-\!b_1(s)x_1(s)\!-\!\frac{c_1(s)x_2(s)}{m(s)+x_1(s)}\right]\!-\!\frac{\sigma_1^2(s)x_1^2(s)}{2(m+x_1(s))^2}\nonumber\\ \left.
+\int\limits_{\R}\left[\rule{0pt}{14pt}\ln\left(\rule{0pt}{12pt}1+\frac{\gamma_1(s,z)x_1(s)}{m+x_1(s)}\right)
-\frac{\gamma_1(s,z)x_1(s)}{m+x_1(s)}\right]\Pi_1(dz)\right\}ds\nonumber\\ +
\int_0^{t}e^s\frac{\sigma_1(s)x_1(s)}{m+x_1(s)}dw_1(s)+\int\limits_0^{t}\!\!\int\limits_{\R}e^s\ln\left(\rule{0pt}{12pt}1+
\frac{\gamma_1(s,z)x_1(s)}{m+x_1(s)}\right)\tilde\nu_1(ds,dz)\nonumber
\end{align}
\begin{align}
+\int\limits_0^{t}\!\!\int\limits_{\R}e^s\ln\left(\rule{0pt}{12pt}1+
\frac{\delta_1(s,z)x_1(s)}{m+x_1(s)}\right)\nu_2(ds,dz).
\end{align}
Let us denote for $0<\kappa\le1$ the process
\begin{align}
\zeta_{\kappa}(t)=\int_0^{t}\!\!\!e^s\frac{\sigma_1(s)x_1(s)}{m+x_1(s)}dw_1(s)\!+\!
\int\limits_0^{t}\!\!\!\int\limits_{\R}\!e^s\ln\left(\rule{0pt}{12pt}1+\frac{\gamma_1(s,z)x_1(s)}{m+x_1(s)}\right)\tilde\nu_1(ds,dz)\nonumber\\
+\!\int\limits_0^{t}\!\!\!\int\limits_{\R}\!e^s\ln\left(\rule{0pt}{12pt}1\!+\!\frac{\delta_1(s,z)x_1(s)}{m+x_1(s)}\right)\!\nu_2(ds,dz)\!-\!
\frac{\kappa}{2}\!\int_0^{t}\!\!e^{2s}\sigma_1^2(s)\left(\frac{x_1(s)}{m+x_1(s)}\right)^2\!\!ds\nonumber\\
-\frac{1}{\kappa}\int\limits_0^{t}\!\!\!\int\limits_{\R}\!\left[\left(1\!+\!\frac{\gamma_1(s,z)x_1(s)}{m+x_1(s)}\right)^{\kappa e^{s}}\!\!\!\!\!-\!1\!-\!\kappa e^s\ln\left(1\!+\!\frac{\gamma_1(s,z)x_1(s)}{m+x_1(s)}\right)\right]\Pi_1(dz)ds\nonumber\\
-\frac{1}{\kappa}\int\limits_0^{t}\!\!\!\int\limits_{\R}\!\left[\left(1\!+\!\frac{\delta_1(s,z)x_1(s)}{m+x_1(s)}\right)^{\kappa e^{s}}-1\right]\Pi_2(dz)ds.\nonumber
\end{align}
By virtue of the exponential inequality (\cite{Bor2}, Lemma 2.2) for any $T>0$,\linebreak $0<\kappa\le1$, $\beta>0$ we have
\begin{align}\label{eq11}
\pr\{\sup_{0\le t\le T}\zeta_{\kappa}(t)>\beta\}\le e^{-\kappa\beta}.
\end{align}
Choose $T=k\tau, k\in \mathbb{N}, \tau>0, \kappa=\varepsilon e^{-k\tau},\beta=\theta e^{k\tau}\varepsilon^{-1}\ln k$, $0<\varepsilon<1$, $\theta>1$ we get
\begin{align}
\pr\{\sup_{0\le t\le T}\zeta_{\kappa}(t)>\theta e^{k\tau}\varepsilon^{-1}\ln k\}\le \frac{1}{k^{\theta}}.\nonumber
\end{align}
By Borel-Cantelli lemma for almost all $\omega\in\Omega$, there is a random integer $k_0(\omega)$, such that $\forall k\ge k_0(\omega)$ and $0\le t\le k\tau$
we have
\begin{align}\label{eq12}
\int_0^{t}\!\!\!e^s\frac{\sigma_1(s)x_1(s)}{m+x_1(s)}dw_1(s)\!+\!
\int\limits_0^{t}\!\!\!\int\limits_{\R}\!e^s\ln\left(\rule{0pt}{12pt}1+\frac{\gamma_1(s,z)x_1(s)}{m+x_1(s)}\right)\tilde\nu_1(ds,dz)\nonumber\\
+\!\int\limits_0^{t}\!\!\!\int\limits_{\R}\!e^s\ln\left(\rule{0pt}{12pt}1\!+\!\frac{\delta_1(s,z)x_1(s)}{m+x_1(s)}\right)\!\nu_2(ds,dz)\le\nonumber\\
\frac{\eps }{2e^{k\tau}}\!\!\int_0^{t}\!\!e^{2s}\left(\frac{\sigma_1(s)x_1(s)}{m+x_1(s)}\right)^2\!\!ds
+\frac{e^{k\tau}}{\eps}\int\limits_0^{t}\!\!\!\int\limits_{\R}\!\left[\left(1\!+\!\frac{\gamma_1(s,z)x_1(s)}{m+x_1(s)}\right)^{\eps e^{s-k\tau}}\right.\nonumber\\ \left.-1-\eps e^{s-k\tau}\ln\left(1\!+\!\frac{\gamma_1(s,z)x_1(s)}{m+x_1(s)}\right)\right]\Pi_1(dz)ds\nonumber\\
+\frac{e^{k\tau}}{\eps}\int\limits_0^{t}\!\!\!\int\limits_{\R}\!\left[\left(1\!+\!\frac{\delta_1(s,z)x_1(s)}{m+x_1(s)}\right)^{\eps e^{s-k\tau}}-1\right]\Pi_2(dz)ds +\frac{\theta e^{k\tau}\ln k}{\eps}.
\end{align}
By using the inequality $x^r\le 1+r(x-1)$, $\forall x\ge0$, $0\le r\le1$  for $x=1+\frac{\gamma_1(s,z)x_1(s)}{m+x_1(s)}$, $r=\eps e^{s-k\tau}$, then for $x=1+\frac{\delta_1(s,z)x_1(s)}{m+x_1(s)}$, $r=\eps e^{s-k\tau}$,  we derive the estimates
\begin{align}\label{eq13}
\frac{e^{k\tau}}{\eps}\int\limits_0^{t}\!\!\!\int\limits_{\R}\!\left[\left(1\!+\!\frac{\gamma_1(s,z)x_1(s)}{m+x_1(s)}\right)^{\eps e^{s-k\tau}}-1
\right.\nonumber\\\left.-\eps e^{s-k\tau}\ln\left(1\!+\!\frac{\gamma_1(s,z)x_1(s)}{m+x_1(s)}\right)\right]\Pi_1(dz)ds\nonumber\\ \le
\int\limits_0^{t}\!\!\!\int\limits_{\R}\!e^s\left[\frac{\gamma_1(s,z)x_1(s)}{m+x_1(s)}-\ln\left(1\!+
\!\frac{\gamma_1(s,z)x_1(s)}{m+x_1(s)}\right)\right]\Pi_1(dz)ds,
\end{align}
\begin{align}\label{eq14}
\frac{e^{k\tau}}{\eps}\int\limits_0^{t}\!\!\!\int\limits_{\R}\!\left[\left(1\!+\!\frac{\delta_1(s,z)x_1(s)}{m+x_1(s)}\right)^{\eps e^{s-k\tau}}-1\right]\Pi_2(dz)ds\nonumber\\ \le
\int\limits_0^{t}\!\!\!\int\limits_{\R}\!e^s\frac{\delta_1(s,z)x_1(s)}{m+x_1(s)}\Pi_2(dz)ds.
\end{align}
From $(\ref{eq10})$, by using $(\ref{eq12})$--$(\ref{eq14})$ we get
\begin{align}\label{eq15}
e^{t}\ln(m+x_1(t))\le \ln(m+x_{10})+\int_{0}^{t}e^s\left\{\rule{0pt}{20pt}\ln(m+x_1(s))\right.\nonumber\\+\frac{x_1(s)}{m+x_1(s)}\left[\rule{0pt}{14pt}
a_1(s)\!-\!b_1(s)x_1(s)\!-\!\frac{c_1(s)x_2(s)}{m(s)+x_1(s)}\right]\!-\!\frac{\sigma_1^2(s)x_1^2(s)}{2(m+x_1(s))^2}\nonumber\\ \left.
\times\left(1-\eps e^{s-k\tau}\right)+\int\limits_{\R}\frac{\delta_1(s,z)x_1(s)}{m+x_1(s)}\Pi_2(dz)\right\}ds+\frac{\theta e^{k\tau}\ln k}{\eps},\ \hbox{a.s.}
\end{align}
It is easy to see that, under Assumption \ref{ass1}, for any $x>0$ there exists a constant $L>0$ independent on $k,s$ and $x$, such that
\begin{align}
\ln(m+x)-\frac{x^2 b_1(s)}{m+x}+\frac{x \alpha_1(s)}{m+x}\le L.\nonumber
\end{align}
So, from $(\ref{eq15})$ for any $(k-1)\tau\le t\le k\tau$ we have (a.s.)
\begin{align}
\frac{\ln(m+x_1(t))}{\ln t}\le e^{-t}\frac{\ln(m+x_{10})}{\ln t}+\frac{L}{\ln t}(1-e^{-t})+\frac{\theta e^{k\tau}\ln k}{\eps e^{(k-1)\tau}\ln(k-1)\tau}.\nonumber
\end{align}
Therefore
\begin{align}
\limsup_{t\to\infty}\frac{\ln(m+x_1(t))}{\ln t}\le \frac{\theta e^{\tau}}{\eps},\ \forall \theta>1, \tau>0, \eps\in (0,1),\quad  \hbox{a.s.}\nonumber
\end{align}
If $\theta\downarrow1, \tau\downarrow 0, \eps\uparrow 1$, then we obtain
\begin{align}
\limsup_{t\to\infty}\frac{\ln(m+x_1(t))}{\ln t}\le1,\quad \hbox{a.s.}
\nonumber
\end{align}
So
\begin{align}
\limsup_{t\to\infty}\frac{\ln(m+x_1(t))}{t}\le0,\quad \hbox{a.s.}
\nonumber
\end{align}
\end{proof}
\begin{cor}\label{cor1}
The density of prey population $x_1(t)$ obeys
\begin{align}
\limsup_{t\to\infty}\frac{\ln x_1(t)}{t}\le0,\quad \hbox{a.s.}\nonumber
\end{align}
\end{cor}

\begin{lemma}\label{lm2}
The density of predator population $x_2(t)$ has the property that
\begin{align}
\limsup_{t\to\infty}\frac{\ln x_2(t)}{t}\le0,\quad \hbox{a.s.}\nonumber
\end{align}
\end{lemma}
\begin{proof}
Making use of It\^{o} formula we get
\begin{align}\label{eq16}
e^{t}\ln x_2(t)-\ln x_{20}=\int_{0}^{t}e^s\left\{\rule{0pt}{16pt}\ln x_2(s)+a_2(s)\!-\!\frac{c_2(s)x_2(s)}{m(s)+x_1(s)}\right.\!-\!\frac{\sigma_2^2(s)}{2}\nonumber\\ \left.
+\int\limits_{\R}\left[\rule{0pt}{14pt}\ln\left(\rule{0pt}{12pt}1+\gamma_2(s,z)\right)-\gamma_2(s,z)\right]\Pi_1(dz)\right\}ds+\psi(t),
\end{align}
where
\begin{align}
\psi(t)=\int_0^{t}e^s\sigma_2(s)dw_2(s)+\int\limits_0^{t}\!\!\int\limits_{\R}e^s\ln\left(\rule{0pt}{12pt}1+
\gamma_2(s,z)\right)\tilde\nu_1(ds,dz)\nonumber\\
+\int\limits_0^{t}\!\!\int\limits_{\R}e^s\ln\left(\rule{0pt}{12pt}1+
\delta_2(s,z)\right)\nu_2(ds,dz).\nonumber
\end{align}
By virtue of the exponential inequality $(\ref{eq11})$ we have
\begin{align}
\pr\{\sup_{0\le t\le T}\zeta_{\kappa}(t)>\beta\}\le e^{-\kappa\beta}, \forall 0<\kappa\le1, \beta>0,\nonumber
\end{align}
where
\begin{align}
\zeta_{\kappa}(t)=\psi(t)-\frac{\kappa}{2}\int\limits_0^t e^{2s}\sigma_2^2(s)ds-
\frac{1}{\kappa}\int\limits_0^{t}\!\!\int\limits_{\R}\left[(1+\gamma_2(s,z))^{\kappa e^s}-1\right.\nonumber\\\left.-\kappa e^s \ln(1+\gamma_2(s,z))\right]\Pi_1(dz)ds-\frac{1}{\kappa}\int\limits_0^{t}\!\!\int\limits_{\R}\left[(1+\delta_2(s,z))^{\kappa e^s}-1\right]\Pi_2(dz)ds.
\nonumber
\end{align}
Choose $T=k\tau, k\in \mathbb{N}, \tau>0, \kappa=e^{-k\tau},\beta=\theta e^{k\tau}\ln k$,  $\theta>1$ we get
\begin{align}
\pr\{\sup_{0\le t\le T}\zeta_{\kappa}(t)>\theta e^{k\tau}\ln k\}\le \frac{1}{k^{\theta}}.\nonumber
\end{align}
By the same arguments as in the proof of Lemma \ref{lm1}, using Borel-Cantelli lemma, we derive from $(\ref{eq16})$
\begin{align}\label{eq17}
e^{t}\ln x_2(t)\le\ln x_{20}+\int_{0}^{t}e^s\left\{\rule{0pt}{16pt}\ln x_2(s)+a_2(s)\!-\!\frac{c_2(s)x_2(s)}{m(s)+x_1(s)}\right.\nonumber\\ \left.
-\frac{\sigma_2^2(s)}{2}\left(1-e^{s-k\tau}\right)
+\int\limits_{\R}\delta_2(s,z)\Pi_2(dz)\right\}ds+\theta e^{k\tau}\ln k,\quad\hbox{a.s.}
\end{align}
for all sufficiently large $k\ge k_0(\omega)$ and $0\le t\le k\tau$.

Using inequality $\ln x -cx\le -\ln c-1$, $\forall x\ge0$, $c>0$ for $x=x_2(s)$, $c=\frac{c_2(s)}{m(s)+x_1(s)}$, we derive from $(\ref{eq17})$ the estimate
\begin{align}
e^{t}\ln x_2(t)\le\ln x_{20}+\int\limits_0^t e^s\ln\left(\rule{0pt}{12pt}m_{\sup}+x_1(s)\right)ds+L(e^t-1)+\theta e^{k\tau}\ln k,\nonumber
\end{align}
for some constant $L>0$.

So for $(k-1)\tau\le t\le k\tau$, $k\ge k_0(\omega)$ we have
\begin{align}
\limsup_{t\to\infty}\frac{\ln x_2(t)}{t}\le \limsup_{t\to\infty}\frac{1}{t}\int\limits_0^t e^{s-t}\ln\left(\rule{0pt}{12pt}m_{\sup}+x_1(s)\right)ds\le0,\nonumber
\end{align}
by virtue of Lemma \ref{lm1}.
\end{proof}

\begin{lemma}\label{lm3}
Let $p > 0$. Then for any initial value $x_{10}>0$, the $p$th-moment of prey population density $x_1(t)$ obeys
\begin{align}\label{eq18}
\limsup_{t\to\infty}\E\left[x_1^p(t)\right] \le K_1(p),
\end{align}
where $K_1(p)>0$ is independent of $x_{10}$.

For any initial value $x_{20}>0$, the expectation  of predator population density $x_2(t)$ obeys
\begin{align}\label{eq19}
\limsup_{t\to\infty}\E\left[x_2(t)\right] \le K_2,
\end{align}
where $K_2>0$ is independent of $x_{20}$.
\end{lemma}
\begin{proof}
Let $\tau_n$ be the stopping time defined in Theorem \ref{thm1}. Applying the It\^{o} formula to the process $V(t,x_1(t))=e^tx_1^p(t)$, $p>0$ we obtain
\begin{align}
V(t\wedge\tau_n,x_1(t\wedge\tau_n))=x_{10}^p+\int\limits_0^{t\wedge\tau_n}e^sx_1^p(s)\left\{\rule{0pt}{18pt}1
+p\left[\rule{0pt}{16pt}a_1(s)-b_1(s)x_1(s)\right.\right.\nonumber
\end{align}
\begin{align}\label{eq20} \left.-\frac{c_{1}(s)x_{2}(s)}{m(s)+x_{1}(s)}\right]
\!\!+\!\frac{p(p-1)\sigma_1^2(s)}{2}\!+\!\!\int\limits_{\R}\!\!\!\left[(1+\gamma_1(s,z))^p\!-\!1\!-\!p\gamma_1(s,z)\right]\Pi_1(dz)\!\nonumber\\\left.
+\!\int\limits_{\R}\!\!\!\left[(1+\delta_1(s,z))^p\!-\!1\right]\Pi_2(dz) \right\}ds
+\int\limits_0^{t\wedge\tau_n}\!\!\!p e^sx_1^p(s)\sigma_1(s)dw_1(s)\nonumber
\\+\int\limits_0^{t\wedge\tau_n}\!\!\!\int\limits_{\R}e^sx_1^p(s)\left[(1+\gamma_1(s,z))^p-1\right]\tilde\nu_1(ds,dz)\nonumber\\
+\int\limits_0^{t\wedge\tau_n}\!\!\!\int\limits_{\R}e^sx_1^p(s)\left[(1+\delta_1(s,z))^p-1\right]\tilde\nu_2(ds,dz).
\end{align}

Under Assumption \ref{ass1} there is constant $K_1(p)>0$, such that
\begin{align}\label{eq21}
e^sx_1^p\left\{\rule{0pt}{18pt}1+p\left[a_1(s)-b_1(s)x_1-\frac{c_{1}(s)x_{2}}{m(s)+x_1}\right]\!+\!\frac{p(p-1)\sigma_1^2(s)}{2}+\right.\nonumber\\ \left.+\!\!\int\limits_{\R}\!\!\left[(1\!+\!\gamma_1(s,z))^p\!-\!1\!-\!p\gamma_1(s,z)\right]\Pi_1(dz)
\!+\!\!\int\limits_{\R}\!\!\left[(1+\delta_1(s,z))^p\!-\!1\right]\Pi_2(dz) \right\}\nonumber\\
\le K_1(p)e^s.
\end{align}
From $(\ref{eq20})$ and $(\ref{eq21})$, taking expectations, we obtain
\begin{align}
\E[V(t\wedge\tau_n,x_1(t\wedge\tau_n))]\le x_{10}^p+K_1(p)e^t.\nonumber
\end{align}

Letting $n\to\infty$ leads to the estimate
\begin{align}\label{eq22}
e^t\E[x_1^p(t)]\le x_{10}^p+e^tK_1(p).
\end{align}
So from $(\ref{eq22})$ we derive $(\ref{eq18})$.

Let us prove the estimate $(\ref{eq19})$. Applying the It\^{o} formula to the process $U(t,X(t))=e^t[k_1x_1(t)+k_2 x_2(t)]$, $k_i>0, i=1,2$ we obtain
\begin{align}\label{eq23}
dU(t,X(t))=e^t\left\{\rule{0pt}{18pt}k_1x_1(t)+k_2x_2(t)+k_1\left[\rule{0pt}{16pt}a_1(t)x_1(t)-b_1(t)x_1^2(t)\right.\right.\nonumber\\ \left.-\frac{c_{1}(t)x_1(t)x_{2}(t)}{m(t)+x_{1}(t)}\right]+k_2\left[\rule{0pt}{16pt}a_2(t)x_2(t)-\frac{c_{2}(t)x^2_{2}(t)}{m(t)+x_{1}(t)}\right]
\nonumber\\\left.
+\!\sum_{i=1}^2k_i\int\limits_{\R}\!\!\!x_i(t)\delta_i(t,z)\Pi_2(dz)\right\}dt
+ e^t\left\{\sum_{i=1}^2k_i\left[\rule{0pt}{18pt}x_i(t)\sigma_i(t)dw_i(t)\right.\right.\nonumber
\\\left. \left.+\int\limits_{\R}x_i(t)\gamma_i(t,z)\tilde\nu_1(dt,dz)+\int\limits_{\R}x_i(t)\delta_i(t,z)\tilde\nu_2(dt,dz)\right]\right\}.
\end{align}
For the function
\begin{align}
f(t,x_1,x_2)\!=\!\frac{1}{m(t)+x_1}\!\left\{\rule{0pt}{13pt}\!k_1\!\left[\!-b_1(t)x_1^3\!+\!\left(\rule{0pt}{12pt}1\!+\!a_1(t)\!+\!
\bar\delta_1(t)\!-\!b_1(t)m(t)\right)x_1^2\right.\right.\nonumber\\ \left.
+m(t)\left(\rule{0pt}{12pt}1\!+\!a_1(t)\!+\!\bar\delta_1(t)\right)x_1\right]+\left[\rule{0pt}{12pt}k_2\left(\rule{0pt}{12pt}1+a_2(t)+\bar\delta_2(t)\right)
-k_1c_1(t)\right]x_1x_2\nonumber\\\left.
+k_2\left[-c_2(t)x_2^2+m(t)\left(\rule{0pt}{12pt}1+a_2(t)+\bar\delta_2(t)\right)x_2\right]\rule{0pt}{13pt}\right\},\nonumber\\ \hbox{where}\
\bar\delta_i(t)=\int\limits_{\R}\!\!\!\delta_i(t,z)\Pi_2(dz),\ i=1,2\nonumber
\end{align}
we have
\begin{align}
f(t,x_1,x_2)\le\frac{\phi_1(x_1,x_2)+\phi_2(x_2)}{m(t)+x_1},\nonumber
\end{align}
where
\begin{align}
\phi_1(x_1,x_2)=k_1\!\left[\!-b_{1\inf}x_1^3\!+\!\left(\rule{0pt}{12pt}d_1-\!b_{1\inf}m_{\inf}\right)x_1^2
+m_{\sup}d_1x_1\right]\nonumber\\+\left[\rule{0pt}{12pt}k_2d_2-k_1c_{1\inf}\right]x_1x_2,\nonumber\\
\phi_2(x_2)=k_2\left[-c_{2\inf}x_2^2+m_{\sup}d_2x_2\right],\
d_i=1\!+\!a_{i\sup}+|\bar\delta_i|_{\sup},\ i=1,2.\nonumber
\end{align}
For $k_2=k_1 c_{1\inf}/d_2$ there is a constant $L'>0$, such that $\phi_1(x_1,x_2)\le L'k_1$ and $\phi_2(x_2)\le L'k_1$. So there is a constant $L>0$, such that
\begin{align}\label{eq24}
f(t,x_1,x_2)\le L k_1.
\end{align}
From $(\ref{eq23})$ and $(\ref{eq24})$ by integrating and taking expectation, we derive
\begin{align}
\E[U(t\wedge\tau_n,X(t\wedge\tau_n))]\le k_1\left[x_{10}+\frac{c_{1\inf}}{d_2} x_{20}+L e^{t}\right].\nonumber
\end{align}
Letting $n\to\infty$ leads to the estimate
\begin{align}
e^t\E\left[x_1(t)+\frac{c_{1\inf}}{d_2}x_2(t)\right]\le x_{10}+\frac{c_{1\inf}}{d_2}x_{20}+L e^{t}.\nonumber
\end{align}
So
\begin{align}\label{eq25}
\E[x_2(t)]\le \left(\frac{d_2}{c_{1\inf}}x_{10}+x_{20}\right)e^{-t}+\frac{d_2}{c_{1\inf}}L.
\end{align}
From $(\ref{eq25})$ we have $(\ref{eq19})$.
\end{proof}
\begin{lemma}\label{lm4}
If $p_{2\inf}>0$, where $p_2(t)=a_2(t)-\beta_2(t)$, then for any initial value $x_{20}>0$, the predator population density $x_2(t)$  satisfies 
\begin{align}\label{eq26}
\limsup_{t\to\infty}\E\left[\left(\frac{1}{x_2(t)}\right)^{\theta}\right] \le K(\theta),\ 0<\theta<1,
\end{align}

\end{lemma}

\begin{proof}
For the process $U(t)=1/x_2(t)$ by the It\^{o} formula we derive
\begin{align}
U(t)=U(0)+\int\limits_0^t U(s) \left[\rule{0pt}{20pt}\frac{c_2(s)x_{2}(s)}{m(s)+x_{1}(s)}-a_2(s)+\sigma_2^2(s)\right.\nonumber\\ \left.+\int\limits_{\R}\frac{\gamma_2^2(s,z)}{1+\gamma_2(s,z)}\Pi_1(dz)\right]ds
-\int\limits_0^tU(s)\sigma_2(s)dw_2(s)\nonumber\\-\int\limits_0^t\!\!\!\int\limits_{\R}U(s)\frac{\gamma_2(s,z)}{1+\gamma_2(s,z)}\tilde\nu_1(ds,dz)-
\int\limits_0^t\!\!\!\int\limits_{\R}U(s)\frac{\delta_2(s,z)}{1+\delta_2(s,z)}\nu_2(ds,dz).
\nonumber
\end{align}
Then, by applying It\^{o} formula, we derive, for $0<\theta<1$
\begin{align}
(1+U(t))^{\theta}=(1+U(0))^{\theta}+\int\limits_0^t \theta(1+U(s))^{\theta-2}\left\{\rule{0pt}{20pt}(1+U(s))U(s)\right.\nonumber\\ \times\left[\frac{c_2(s)x_{2}(s)}{m(s)+x_{1}(s)}-a_2(s)+\sigma_2^2(s)\!
+\!\int\limits_{\R}\frac{\gamma_2^2(s,z)}{1+\gamma_2(s,z)}\Pi_1(dz)\right]
\nonumber
\end{align}
\begin{align}
+\frac{\theta-1}{2}U^2(s)\sigma_2^2(s)\nonumber\\+\frac{1}{\theta}\int\limits_{\R}\left[(1+U(s))^2\left(\left(\frac{1+U(s)+
\gamma_2(s,z)}{(1+\gamma_2(s,z))(1+U(s))}\right)^{\theta}-1\right)\right.\nonumber\\+\left.
\theta(1+U(s))\frac{U(s)\gamma_2(s,z)}{1+\gamma_2(s,z)}\right]\Pi_1(dz)\nonumber\end{align}
\begin{align} \left.
+\frac{1}{\theta}\int\limits_{\R}(1+U(s))^2\left[\left(\frac{1+U(s)+\delta_2(s,z)}{(1+\delta_2(s,z))(1+U(s))}\right)^{\theta}
-1\right]\Pi_2(dz)\right\}ds\nonumber\\
-\int\limits_0^t\theta(1+U(s))^{\theta-1}U(s)\sigma_2(s)dw_2(s)\nonumber\\+
\int\limits_0^t\!\!\int\limits_{\R}\left[\left(1+\frac{U(s)}{1+\gamma_2(s,z)}\right)^{\theta}-(1+U(s))^\theta\right]\tilde\nu_1(ds,dz)\nonumber
\end{align}
\begin{align}
+\int\limits_0^t\!\!\int\limits_{\R}\left[\left(1+\frac{U(s)}{1+\delta_2(s,z)}\right)^{\theta}-(1+U(s))^\theta\right]\tilde\nu_2(ds,dz)
\nonumber
\end{align}
\begin{align}\label{eq27}
=(1+U(0))^{\theta} +\int\limits_0^t \theta(1+U(s))^{\theta-2}J(s)ds\nonumber\\-I_{1,stoch}(t)+I_{2,stoch}(t)+I_{3,stoch}(t),
\end{align}
where $I_{j,stoch}(t), j=\overline{1,3}$ are the corresponding stochastic integrals in $(\ref{eq27})$. Under the Assumption \ref{ass1} there exists constants $|K_1(\theta)|<\infty$, $|K_2(\theta)|<\infty$ such, that for the process $J(t)$ we have the estimate
\begin{align}
J(t)\le (1+U(t))U(t)\left[-a_{2}(t)+\frac{c_{2\sup}U^{-1}(t)}{m_{inf}}
+\sigma_2^2(t)\!\!\phantom{\int\limits_{\R}}\right.\nonumber\\ \left.+\int\limits_{\R}\frac{\gamma_2^2(s,z)}{1+\gamma_2(s,z)}\Pi_1(dz)\right]
+\frac{\theta-1}{2}U^2(s)\sigma_2^2(s)\nonumber\\ +\frac{1}{\theta}\int\limits_{\R}\left[(1+U(s))^2\left(\left(\frac{1}{1+\gamma_2(s,z)}+
\frac{1}{1+U(s)}\right)^{\theta}-1\right)\right.\nonumber\\+\left.
\theta(1+U(s))\frac{U(s)\gamma_2(s,z)}{1+\gamma_2(s,z)}\right]\Pi_1(dz)
\nonumber\\
+\frac{1}{\theta}\int\limits_{\R}(1+U(s))^2\left[\left(\frac{1}{1+\delta_2(s,z)}+
\frac{1}{1+U(s)}\right)^{\theta}-1\right]\Pi_2(dz)\nonumber\\ \le U^2(t)\left[-a_{2}(t)+\frac{\sigma_2^2(t)}{2} +\int\limits_{\R}\gamma_2(t,z)\Pi_1(dz)+\frac{\theta}{2}\sigma_2^2(t)\right.\nonumber\\ \left.+\frac{1}{\theta}\int\limits_{\R}[(1+\gamma_2(t,z))^{-\theta}-1]\Pi_1(dz)
+\frac{1}{\theta}\int\limits_{\R}[(1+\delta_2(t,z))^{-\theta}-1]\Pi_2(dz)\right]\nonumber\\ +K_1(\theta)U(t)+K_2(\theta)=-K_0(t,\theta)U^2(t)+K_1(\theta)U(t)+K_2(\theta),
\nonumber
\end{align}
where we used the inequality $(x+y)^{\theta}\le x^{\theta}+\theta x^{\theta-1}y$, $0<\theta<1$, $x,y>0$. Due to
\begin{align}
\lim_{\theta\to0+}\left[\frac{\theta}{2}\sigma_2^2(t)+\frac{1}{\theta}\int\limits_{\R}[(1+\gamma_2(t,z))^{-\theta}-1]\Pi_1(dz)\right.\nonumber\\ \left.
+\frac{1}{\theta}\int\limits_{\R}[(1+\delta_2(t,z))^{-\theta}-1]\Pi_2(dz)
+\int\limits_{\R}\ln(1+\gamma_2(t,z))\Pi_1(dz)\right.\nonumber\\\left.+\int\limits_{\R}\ln(1+\delta_2(t,z))\Pi_2(dz)\right]=\lim_{\theta\to0+}\Delta(\theta)=0,\nonumber
\end{align}
and condition $p_{2\inf}>0$ we can choose a sufficiently small $0<\theta<1$ to satisfy
\begin{align}
K_0(\theta)=\inf_{t\ge0}K_0(t,\theta)=\inf_{t\ge0}[p_2(t)-\Delta(\theta)]= p_{2\inf}-\Delta(\theta)>0.\nonumber
\end{align}
So from $(\ref{eq27})$ and the estimate for $J(t)$ we derive
\begin{align}\label{eq28}
d\left[(1+U(t))^\theta\right]\le\theta(1+U(t))^{\theta-2}[-K_0(\theta)U^2(t)+K_1(\theta)U(t)+K_2(\theta)]dt\nonumber\\
-\theta(1+U(t))^{\theta-1}U(t)\sigma_2(t)dw_2(t)+
\int\limits_{\R}\left[\left(1+\frac{U(t)}{1+\gamma_2(t,z)}\right)^{\theta}\right.
\nonumber\\ \left.
-(1\!+\!U(t))^\theta\rule{0pt}{18pt}\right]\!\tilde\nu_1(dt,dz)\! +\!
\int\limits_{\R}\!\left[\left(1\!+\!\frac{U(t)}{1+\delta_2(t,z)}\right)^{\theta}\!\!
-\!(1\!+\!U(t))^\theta\right]\!\tilde\nu_2(dt,dz).
\end{align}
By the It\^{o} formula and $(\ref{eq28})$ we have
\begin{align}\label{eq29}
d\left[e^{\la t}(1+U(t))^\theta\right]=\la e^{\la t}(1+U(t))^\theta dt+e^{\la t} d\left[(1+U(t))^\theta\right]\nonumber \\
\le e^{\la t}\theta(1+U(t))^{\theta-2}\left[-\left(K_0(\theta)-\frac{\la}{\theta}\right)U^2(t)+\left(K_1(\theta)+\frac{2\la}{\theta}\right)U(t)\right.\nonumber\\ \left.
+K_2(\theta)+\frac{\la}{\theta}\right]dt-\theta e^{\la t}(1+U(t))^{\theta-1}U(t)\sigma_2(t)dw_2(t)\nonumber\\
+e^{\la t}\int\limits_{\R}\left[\left(1+\frac{U(t)}{1+\gamma_2(t,z)}\right)^{\theta}-(1+U(t))^\theta\right]\tilde\nu_1(dt,dz)\nonumber\\ \displaystyle+
e^{\la t}\int\limits_{\R}\left[\left(1+\frac{U(t)}{1+\delta_2(t,z)}\right)^{\theta}-(1+U(t))^\theta\right]\tilde\nu_2(dt,dz).
\end{align}

Let us choose $\la=\la(\theta)>0$ such that $K_0(\theta)-\la/\theta>0$. Then there is a constant $K>0$, such that
\begin{align}\label{eq30}
(1+U(t))^{\theta-2}\left[-\left(K_0(\theta)-\frac{\la}{\theta}\right)U^2(t)\right.\nonumber\\\left.
+\left(K_1(\theta)+\frac{2\la}{\theta}\right)U(t)
+K_2(\theta)+\frac{\la}{\theta}\right]\le K.
\end{align}
Let $\tau_n$ be the stopping time defined in Theorem \ref{thm1}. Then by integrating $(\ref{eq29})$, using $(\ref{eq30})$  and taking the expectation we obtain
\begin{align}
\E\left[e^{\la (t\wedge\tau_n)}(1+U(t\wedge\tau_n))^\theta\right]\le \left(1+\frac{1}{x_{20}}\right)^\theta+\frac{\theta}{\la}K\left(e^{\la t}-1\right).
\nonumber
\end{align}
Letting $n\to\infty$ leads to the estimate
\begin{align}\label{eq31}
e^t\E\left[(1+U(t))^\theta\right]\le \left(1+\frac{1}{x_{20}}\right)^\theta+\frac{\theta}{\la}K\left(e^{\la t}-1\right).
\end{align}
From $(\ref{eq31})$ we obtain
\begin{align}
\limsup_{t\to\infty}\E\left[\left(\frac{1}{x_2(t)}\right)^{\theta}\right]=\limsup_{t\to\infty}\E\left[U^{\theta}(t)\right]\nonumber\\ \le \limsup_{t\to\infty}\E\left[(1+U(t))^{\theta}\right]\le \frac{\theta}{\la(\theta)}K,\nonumber
\end{align}
this implies  $(\ref{eq26})$.
\end{proof}

\section{The long time behaviour}

\begin{defin}\label{def1}(\cite{LiMao})
The solution $X(t)$ to the system $(\ref{eq3})$ is said to be stochastically ultimately
bounded, if for any $\varepsilon\in(0,1)$, there is a positive constant $\chi=\chi(\varepsilon)>0$, such that for
any initial value $X_0\in\R^2_{+}$, the solution to the system $(\ref{eq3})$ has the property that
\begin{align}
\limsup_{t\to\infty}\pr\left\{|X(t)|>\chi\right\}<\varepsilon.\nonumber
\end{align}
\end{defin}

In what follows in this section we will assume that Assumption \ref{ass1} holds.

\begin{thm}\label{thm2}
The solution $X(t)$ to the system $(\ref{eq3})$ with initial value $X_0\in\R^2_{+}$ is stochastically ultimately
bounded.
\end{thm}

\begin{proof}

From the Lemma \ref{lm3} we have the estimate
\begin{align}\label{eq32}
\limsup_{t\to\infty}E[x_i(t)]\le K_i,\ i=1,2.
\end{align}
For $X=(x_1,x_2)\in\R^2_{+}$ we have $|X|\le x_1+x_2$, therefore, from $(\ref{eq32})$\linebreak
$\limsup_{t\to\infty}E[|X(t)|]\le L=K_1+K_2$. Let $\chi>L/\varepsilon$, $\forall\varepsilon\in(0,1)$. Then applying the Chebyshev inequality yields
\begin{align}
\limsup_{t\to\infty}\pr\{|X(t)|>\chi\}\le\frac{1}{\chi}\lim\sup_{t\to\infty}E[|X(t)|]\le \frac{L}{\chi}<\varepsilon.\nonumber
\end{align}
\end{proof}

The property of stochastic permanence is important since it means the long-time survival in a population dynamics.

\begin{defin}
The population density $x(t)$ is said to be stochastically permanent if for any $\eps>0$, there are positive constants $H=H(\eps)$, $h=h(\eps)$ such that
\begin{align}
\liminf\limits_{t\to\infty}\pr\{x(t)\le H\}\ge1-\eps,\quad \liminf\limits_{t\to\infty}\pr\{x(t)\ge h\}\ge1-\eps,
\nonumber
\end{align}
for any inial value $x_{0}>0$.
\end{defin}

\begin{thm}
 If $p_{2\inf}>0$, where $p_2(t)=a_2(t)-\beta_2(t)$, then for any initial value $x_{20}>0$, the predator population density $x_2(t)$ is stochastically permanent.
\end{thm}

\begin{proof}
From Lemma \ref{lm3} we have estimate
\begin{align}
\limsup_{t\to\infty}E[x_2(t)]\le K.\nonumber
\end{align}
Thus for any given $\eps > 0$, let $H = K/\eps$, by virtue of Chebyshev's inequality, we can derive that
\begin{align}
\limsup\limits_{t\to\infty}\pr\{x_2(t)\ge H\}\le \frac{1}{H}\limsup_{t\to\infty}E[x_2(t)]\le\eps.\nonumber
\end{align}
Consequently $\liminf\limits_{t\to\infty}\pr\{x_2(t)\le H\}\ge1-\eps$.

From Lemma \ref{lm4} we have estimate
\begin{align}
\limsup_{t\to\infty}\E\left[\left(\frac{1}{x_2(t)}\right)^{\theta}\right] \le K(\theta),\ 0<\theta<1.\nonumber
\end{align}
For any given $\eps > 0$, let $h = (\eps/K(\theta))^{1/\theta}$, then by Chebyshev's inequality, we have
\begin{align}
\limsup\limits_{t\to\infty}\pr\{x_2(t)< h\}\le \limsup\limits_{t\to\infty}\pr\left\{\left(\frac{1}{x_2(t)}\right)^{\theta}>h^{-\theta}\right\}\nonumber\\\le
h^\theta\limsup\limits_{t\to\infty}\E\left[\left(\frac{1}{x_2(t)}\right)^{\theta}\right]\le\eps.\nonumber
\end{align}
Consequently $\liminf\limits_{t\to\infty}\pr\{x_2(t)\ge h\}\ge1-\eps$.
\end{proof}

\begin{thm}\label{thm3}
 If the predator is absent, i.e. $x_2(t)=0$ a.s., and $p_{1\inf}>0$, where $p_1(t)=a_1(t)-\beta_1(t)$, then for any initial value $x_{10}>0$, the prey population density $x_1(t)$ is stochastically permanent.
\end{thm}

\begin{proof}
From Lemma \ref{lm3} we have estimate
\begin{align}
\limsup_{t\to\infty}E[x_1(t)]\le K.\nonumber
\end{align}
Thus for any given $\eps > 0$, let $H = K/\eps$, by virtue of Chebyshev's inequality, we can derive that
\begin{align}
\limsup\limits_{t\to\infty}\pr\{x_1(t)\ge H\}\le \frac{1}{H}\limsup_{t\to\infty}E[x_1(t)]\le\eps.\nonumber
\end{align}
Consequently $\liminf\limits_{t\to\infty}\pr\{x_1(t)\le H\}\ge1-\eps$.

For the process $U(t)=1/x_1(t)$ by the It\^{o} formula we have
\begin{align}
U(t)=U(0)+\int\limits_0^t U(s) \left[\rule{0pt}{20pt}b_1(s)x_1(s)-a_1(s)+\sigma_1^2(s)\right.\nonumber\\ \left.+\int\limits_{\R}\frac{\gamma_1^2(s,z)}{1+\gamma_1(s,z)}\Pi_1(dz)\right]ds
-\int\limits_0^tU(s)\sigma_1(s)dw_1(s)\nonumber\\-\int\limits_0^t\!\!\!\!\int\limits_{\R}U(s)\frac{\gamma_1(s,z)}{1+\gamma_1(s,z)}\tilde\nu_1(ds,dz)-
\int\limits_0^t\!\!\!\int\limits_{\R}U(s)\frac{\delta_1(s,z)}{1+\delta_1(s,z)}\nu_2(ds,dz).
\nonumber
\end{align}
Then, using the same arguments as in the proof of Lemma \ref{lm4} we can derive the estimate
\begin{align}
\limsup_{t\to\infty}\E\left[\left(\frac{1}{x_1(t)}\right)^{\theta}\right] \le K(\theta),\ 0<\theta<1,\nonumber
\end{align}
For any given $\eps > 0$, let $h = (\eps/K(\theta))^{1/\theta}$, then by Chebyshev's inequality, we have
\begin{align}
\limsup\limits_{t\to\infty}\pr\{x_1(t)< h\}=\limsup\limits_{t\to\infty}\pr\left\{\left(\frac{1}{x_1(t)}\right)^{\theta}>h^{-\theta}\right\}\nonumber\\\le
h^\theta\limsup\limits_{t\to\infty}\E\left[\left(\frac{1}{x_1(t)}\right)^{\theta}\right]\le\eps.\nonumber
\end{align}
Consequently $\liminf\limits_{t\to\infty}\pr\{x_1(t)\ge h\}\ge1-\eps$.
\end{proof}
\begin{remark}
If the predator is absent, i.e. $x_2(t)=0$ a.s., then the equation for the prey $x_1(t)$ has the logistic form. So Theorem \ref{thm3} gives us the sufficient conditions for the stochastic permanence of the solution to the stochastic non-autonomous logistic equation disturbed by white noise, centered and non-centered Poisson noises.
\end{remark}

\begin{defin}
The solution $X(t)=(x_1(t),x_2(t))$, $t\ge0$ to the equation $(\ref{eq3})$ will be said extinct if for every initial data $X_0\in\R^2_{+}$, we have $\lim_{t\to\infty}x_i(t)=0$ almost surely (a.s.), $i=1,2$.
\end{defin}

\begin{thm}\label{thm5}
 If
\begin{align}
{\bar p}^{*}_i=\limsup_{t\to\infty}\frac{1}{t}\int\limits_0^tp_{i}(s)ds<0,\ \hbox{where}\
p_{i}(t)=a_{i}(t)-\beta_{i}(t),\ i=1,2, \nonumber
\end{align}
then the solution $X(t)$ to the equation $(\ref{eq3})$ with initial condition $X_0\in\R^2_{+}$ will be extinct.
\end{thm}

\begin{proof}
By the It\^{o} formula, we have
\begin{align}\label{eq33}
d\ln x_i(t)=\left[a_{i}(t)-b_i(t)x_i(t)-\frac{c_i(t)x_2(t)}{m(t)+x_1(t)}-\beta_i(t)\right]dt+dM_i(t)\nonumber\\ \le
[a_i(t)-\beta_i(t)]dt+dM_i(t),\ i=1,2,
\end{align}
where the martingale
\begin{align}\label{eq34}
M_i(t)=\int\limits_0^t\sigma_i(s)dw_i(s)+\int\limits_0^t\!\!\int\limits_{\R}\ln(1+\gamma_i(s,z))\tilde\nu_1(ds,dz) \nonumber\\+
\int\limits_0^t\!\!\int\limits_{\R}\ln(1+\delta_i(s,z))\tilde\nu_2(ds,dz),\ i=1,2,
\end{align}
has quadratic variation
\begin{align}
\langle M_i,M_i\rangle(t)=\int\limits_0^t\sigma^2_i(s)ds+\int\limits_0^t\!\!\int\limits_{\R}\ln^2(1+\gamma_i(s,z))\Pi_1(dz)ds\nonumber\\+
\int\limits_0^t\!\!\int\limits_{\R}\ln^2(1+\delta_i(s,z))\Pi_2(dz)ds\le Kt, \ i=1,2.\nonumber
\end{align}

Then the strong law of large numbers for local martingales (\cite{Lip}) yields $\lim_{t\to\infty}M_i(t)/t=0, i=1,2$ a.s. Therefore, from $(\ref{eq33})$ we obtain
$$
\limsup_{t\to\infty}\frac{\ln x_i(t)}{t}\le\limsup_{t\to\infty}\frac{1}{t}\int\limits_0^t p_{i}(s)ds<0, \quad\hbox{a.s.}
$$
So $\lim_{t\to\infty}x_i(t)=0, i=1,2$ a.s.
\end{proof}

\begin{defin}[\cite{Liu}]
The population density $x(t)$ will be said non-persistent in the mean if
\begin{align}
\lim_{t\to\infty}\frac{1}{t}\int_0^tx(s)ds=0\ \hbox{a.s.}
\nonumber
\end{align}
\end{defin}

\begin{thm}
If ${\bar p}_1^{*}=0$,
then the prey population density $x_1(t)$ with initial condition $x_{10}>0$ will be non-persistent in the mean.
\end{thm}

\begin{proof}
From the first equality in $(\ref{eq33})$ we have for $i=1$
\begin{align}\label{eq35}
\ln x_1(t)\le \ln x_{10}+\int\limits_0^t p_{1}(s)ds
-b_{1\inf}\int\limits_0^tx_1(s)ds+M_1(t),
\end{align}
where martingale $M_1(t)$ is defined in $(\ref{eq34})$. From the definition of ${\bar p}^{*}_1$ and the strong law of large numbers for $M_1(t)$ it follows, that $\forall \eps>0$, $\exists t_0\ge0$, and $\exists\Omega_{\eps}\subset\Omega$, with $\pr(\Omega_{\eps})\ge1-\eps$, such that
\begin{align}
\frac{1}{t}\int\limits_0^t p_{1}(s)ds\le{\bar p}^{*}_1+\frac{\eps}{2},\ \frac{M_1(t)}{t}\le\frac{\eps}{2},\ \forall t\ge t_0,\ \omega\in\Omega_{\eps}.
\nonumber
\end{align}
So, from $(\ref{eq35})$ we derive
\begin{align}\label{eq36}
\ln x_1(t)- \ln x_{10}\le t({\bar p}^{*}_1+\eps)-b_{1\inf}\int\limits_0^tx_1(s)ds\nonumber\\
=t\eps-b_{1\inf}\int\limits_0^tx_1(s)ds, \forall t\ge t_0,\ \omega\in\Omega_{\eps}.
\end{align}
Let $y_1(t) =\int_0^tx_1(s)ds$, then from $(\ref{eq36})$ we have
\begin{align}
\ln\left(\frac{dy_1(t)}{dt}\right)\le\eps t-b_{1\inf} y_1(t)+\ln x_{10}\nonumber\\ \Rightarrow
e^{b_{1\inf}y_1(t)}\frac{dy_1(t)}{dt}\le x_{10}e^{\eps t}, \forall t\ge t_0,\ \omega\in\Omega_{\eps}.
\nonumber
\end{align}
By integrating of last inequality from $t_0$ to $t$ we obtain
\begin{align}
e^{b_{1\inf}y_1(t)}\le \frac{b_{1\inf}x_{10}}{\eps}\left(e^{\eps t}-e^{\eps t_0}\right)
+e^{b_{1\inf}y_1(t_0)},\ \forall t\ge t_0,\ \omega\in\Omega_{\eps}.
\nonumber
\end{align}
So
\begin{align}
y_1(t)\le \frac{1}{b_{1\inf}}\ln\left[e^{b_{1\inf}y_1(t_0)}+\frac{b_{1\inf}x_{10}}{\eps}\left(e^{\eps t}-e^{\eps t_0}\right)
\right], \ \forall t\ge t_0,\ \omega\in\Omega_{\eps},
\nonumber
\end{align}
and therefore
\begin{align}
\limsup_{t\to\infty}\frac{1}{t}\int\limits_0^tx_1(s)ds\le\frac{\eps}{b_{1\inf}}, \ \forall  \omega\in\Omega_{\eps}.
\nonumber
\end{align}
Since $\eps>0$ is arbitrary and $x_1(t)>0$ a.s., we have
\begin{align}\lim_{t\to\infty}\frac{1}{t}\int\limits_{0}^{t}x_1(s)ds=0\ a.s.\nonumber
\end{align}
\end{proof}

\begin{thm}
If ${\bar p}_2^{*}=0$ and ${\bar p}_1^{*}<0$,
then the predator population density $x_2(t)$ with initial condition $x_{20}>0$ will be non-persistent in the mean.
\end{thm}

\begin{proof}
From the first equality in $(\ref{eq33})$ with $i=2$ we have for $c=c_{2\inf}/m_{\sup}$
\begin{align}\label{eq37}
\ln x_2(t)\le \ln x_{20}+\int\limits_0^t p_{2}(s)ds
-c_{2\inf}\int\limits_0^t\frac{x_2(s)}{m(s)+x_1(s)}ds+M_2(t)\nonumber\\
=\ln x_{20}\!+\!\int\limits_0^t p_{2}(s)ds\!
-\!c_{2\inf}\int\limits_0^t\frac{1}{m(s)}\left[x_2(s)\!-\!\frac{x_1(s)x_2(s)}{m(s)+x_1(s)}\right]ds\!+\!M_2(t)\nonumber
\\ \le\ln x_{20}+\int\limits_0^t p_{2}(s)ds-c\int\limits_0^tx_2(s)ds+c\int\limits_0^t\frac{x_1(s)x_2(s)}{m_{\sup}+x_1(s)}ds+M_2(t),
\end{align}
where martingale $M_2(t)$ is defined in $(\ref{eq34})$. From Theorem \ref{thm5}, the definition of ${\bar p}^{*}_2$ and the strong law of large numbers for $M_2(t)$ it follows, that $\forall \eps>0$, $\exists t_0\ge0$, and $\exists\Omega_{\eps}\subset\Omega$ with $\pr(\Omega_{\eps})\ge1-\eps$, such that
\begin{align}
\frac{1}{t}\int\limits_0^t p_{2}(s)ds\le{\bar p}^{*}_2+\frac{\eps}{2},\ \frac{M_2(t)}{t}\le\frac{\eps}{2},\
\frac{x_1(t)}{m_{\sup}+x_1(t)}\le\eps,\ \forall t\ge t_0,\ \omega\in\Omega_{\eps}.
\nonumber
\end{align}
So, from $(\ref{eq37})$ we derive
\begin{align}\label{eq38}
\ln x_2(t)- \ln x_{20}\le t({\bar p}^{*}_2+\eps)-c(1-\eps)\int\limits_{t_0}^tx_2(s)ds\nonumber\\
=t\eps-c(1-\eps)\int\limits_{t_0}^tx_2(s)ds, \forall t\ge t_0,\ \omega\in\Omega_{\eps}.
\end{align}
Let $y_2(t) =\int_{t_0}^tx_2(s)ds$, then from $(\ref{eq38})$ we have
\begin{align}
\ln\left(\frac{dy_2(t)}{dt}\right)\le\eps t-c(1-\eps) y_2(t)+\ln x_{20}\nonumber\\ \Rightarrow
e^{c(1-\eps)y_2(t)}\frac{dy_2(t)}{dt}\le x_{20}e^{\eps t}, \forall t\ge t_0,\ \omega\in\Omega_{\eps}.
\nonumber
\end{align}
By integrating of last inequality from $t_0$ to $t$ we obtain
\begin{align}
e^{c(1-\eps)y_2(t)}\le \frac{c(1-\eps)x_{20}}{\eps}\left(e^{\eps t}-e^{\eps t_0}\right)
+1,\ \forall t\ge t_0,\ \omega\in\Omega_{\eps}.
\nonumber
\end{align}
So
\begin{align}
y_2(t)\le \frac{1}{c(1-\eps)}\ln\left[1+\frac{c(1-\eps)x_{20}}{\eps}\left(e^{\eps t}-e^{\eps t_0}\right)
\right], \ \forall t\ge t_0,\ \omega\in\Omega_{\eps},
\nonumber
\end{align}
and therefore
\begin{align}
\limsup_{t\to\infty}\frac{1}{t}\int\limits_0^tx_2(s)ds\le\frac{\eps}{c(1-\eps)}, \ \forall  \omega\in\Omega_{\eps}.
\nonumber
\end{align}
Since $\eps>0$ is arbitrary and $x_2(t)>0$ a.s., we have
\begin{align}\lim_{t\to\infty}\frac{1}{t}\int\limits_{0}^{t}x_2(s)ds=0\ a.s.\nonumber
\end{align}
\end{proof}

\begin{defin}[\cite{Liu}]
The population density $x(t)$ will be said weakly persistent in the mean if
\begin{align}
{\bar x}^{*}=\limsup_{t\to\infty}\frac{1}{t}\int_0^tx(s)ds>0\ \hbox{a.s.}
\nonumber
\end{align}
\end{defin}

\begin{thm}\label{thm8}
If ${\bar p}_2^{*}>0$, then the predator population density $x_2(t)$ with initial condition $x_{20}>0$ will be weakly persistent in the mean. 
\end{thm}

\begin{proof}
If the assertion of theorem is not true, then $\pr\{{\bar x}_2^{*}=0\}>0$. From the first equality in $(\ref{eq33})$ we get
\begin{align}
\frac{1}{t}(\ln x_2(t)-\ln x_{20})=\frac{1}{t}\int_0^tp_2(s)ds-\frac{1}{t}\int_0^t\frac{c_2(s)x_2(s)}{m(s)+x_1(s)}ds+\frac{M_2(t)}{t}\nonumber\\
\ge \frac{1}{t}\int_0^t p_2(s)ds-\frac{c_{2\sup}}{m_{\inf}t}\int_0^tx_2(s)ds+\frac{M_2(t)}{t},\nonumber
\end{align}
where martingale $M_2(t)$ is defined in $(\ref{eq34})$. For $\forall\omega\in\{\omega\in\Omega|\ {\bar x}_2^{*}=0\}$ in virtue strong law of large numbers for martingale $M_2(t)$ we have
\begin{align}
\limsup_{t\to\infty}\frac{\ln x_2(t)}{t}\ge {\bar p}_2^{*}>0.\nonumber
\end{align}
Therefore
\begin{align}
\pr\left\{\omega\in\Omega|\ \limsup_{t\to\infty}\frac{\ln x_2(t)}{t}>0\right\}>0.\nonumber
\end{align}
But from Lemma \ref{lm2} we have
\begin{align}
\pr\left\{\omega\in\Omega|\ \limsup_{t\to\infty}\frac{\ln x_2(t)}{t}\le0\right\}=1.\nonumber
\end{align}
This is a contradiction.
\end{proof}

\begin{thm}
If ${\bar p}_1^{*}>0$ and ${\bar p}_2^{*}<0$, then the prey population density $x_1(t)$ with initial condition $x_{10}>0$ will be weakly persistent in the mean.
\end{thm}

\begin{proof}
Let $\pr\{{\bar x}_1^{*}=0\}>0$. From the first equality in $(\ref{eq33})$ with $i=1$ we get
\begin{align}\label{eq39}
\frac{1}{t}(\ln x_1(t)-\ln x_{10})=\frac{1}{t}\int_0^tp_1(s)ds-\frac{1}{t}\int_0^t b_1(s)x_1(s)ds\nonumber\\-\frac{1}{t}\int_0^t\frac{c_1(s)x_2(s)}{m(s)+x_1(s)}ds+\frac{M_1(t)}{t}\nonumber\\
\ge \frac{1}{t}\int_0^tp_1(s)ds-\frac{b_{1\sup}}{t}\int_0^t x_1(s)ds-\frac{c_{1\sup}}{m_{\inf}t}\int_0^tx_2(s)ds+\frac{M_1(t)}{t}
\end{align}
where martingale $M_1(t)$ is defined in $(\ref{eq34})$. From definition of ${\bar p}^{*}_1$, strong law of large numbers for martingale $M_1(t)$ and Theorem \ref{thm2} for $x_2(t)$ we have $\forall\eps>0$, $\exists t_0\ge0$, $\exists \Omega_{\eps}\subset\Omega$ with $\pr(\Omega_{\eps})\ge1-\eps$, such that
\begin{align}
\frac{1}{t}\int\limits_0^t p_{1}(s)ds\ge{\bar p}^{*}_1-\frac{\eps}{3},\ \frac{M_1(t)}{t}\ge-\frac{\eps}{3},\
\frac{1}{t}\int_0^tx_2(s)ds\le\frac{\eps m_{\inf}}{3c_{1\sup}}, \forall t\ge t_0, \omega\in\Omega_{\eps}.
\nonumber
\end{align}
So, from $(\ref{eq39})$ we get for $\omega\in\{\omega\in\Omega|{\bar x}_1^{*}=0\}\cap\Omega_{\eps}$
\begin{align}
\limsup_{t\to\infty}\frac{\ln x_1(t)}{t}\ge {\bar p}_1^{*}-\eps>0
\nonumber
\end{align}
for sufficiently small $\eps>0$. Therefore
\begin{align}
\pr\left\{\omega\in\Omega|\ \limsup_{t\to\infty}\frac{\ln x_1(t)}{t}>0\right\}>0.
\nonumber
\end{align}
But from Corollary \ref{cor1}
\begin{align}
\pr\left\{\omega\in\Omega|\ \limsup_{t\to\infty}\frac{\ln x_1(t)}{t}\le0\right\}=1.
\nonumber
\end{align}
Therefore we have a contradiction.
\end{proof}









\end{document}